\input ./style/arxiv-vmsta.cfg
\documentclass[numbers,compress,v1.0.1]{vmsta}

\usepackage{dsfont}

\volume{3}
\issue{1}
\pubyear{2016}
\firstpage{79}
\lastpage{94}
\doi{10.15559/16-VMSTA52}% Updated by VTEXPTS2LaTeX.exe, 29.03.2016
\startlocaldefs

\allowdisplaybreaks

\newtheorem{thm}{Theorem}
\newtheorem{lemma}{Lemma}
\newtheorem{cor}{Corollary}
\def\ind{\mathds{1}}
\endlocaldefs

\theoremstyle{definition}
\newtheorem{ex}{Example}

\begin{document}
\begin{frontmatter}

\title{Random convolution of inhomogeneous distributions with $\boldsymbol{\mathcal{O}}$-exponential tail}
%\author[]{\inits{}\fnm{}\snm{}\corref{cor1}}\email{}
%\cortext[cor1]{Corresponding author.}
%
%\author[]{\inits{}\fnm{}\snm{}}\email{}
%
%%\fnref{f1}
%%\fntext[]{Some remarks}
%
%\address[]{}
%\address[]{}

\author[a]{\inits{S.}\fnm{Svetlana}\snm{Danilenko}}\email{svetlana.danilenko@vgtu.lt}
\author[b]{\inits{S.}\fnm{Simona}\snm{Pa\v skauskait\. e}}\email{simona.paskauskaite@mif.vu.stud.lt}
\author[b]{\inits{J.}\fnm{Jonas}\snm{\v Siaulys}\corref{cor1}}\email{jonas.siaulys@mif.vu.lt}
\cortext[cor1]{Corresponding author.}

\address[a]{Faculty of Fundamental Sciences, Vilnius
Gediminas Technical University,
Saul\.etekio al. 11, Vilnius LT-10223, Lithuania}
\address[b]{Faculty of Mathematics and Informatics, Vilnius
University,\break Naugarduko 24, Vilnius LT-03225, Lithuania}

\markboth{S. Danilenko et al.}{Random convolution of inhomogeneous distributions with $\mathcal
{O}$-exponential tail}

\begin{abstract} Let $\{\xi_1,\xi_2,\ldots\}$ be a sequence of
independent random variables (not necessarily identically distributed),
and $\eta$ be a counting random variable independent of this sequence.
We obtain sufficient conditions on $\{\xi_1,\xi_2,\ldots\}$ and~$\eta$
under which the distribution function of the random sum $S_\eta=\xi
_1+\xi_2+\cdots+\xi_\eta$ belongs to the class of $\mathcal{O}$-exponential distributions.
\end{abstract}

%\begin{keyword} . \sep.
%\MSC[2010] . \sep.
%\end{keyword}
%

\received{28 January 2016}% Updated by VTEXPTS2LaTeX.exe, 29.03.2016
%14:29
%
\revised{21 March 2016}% Updated by VTEXPTS2LaTeX.exe, 29.03.2016 14:29
\accepted{21 March 2016}% Updated by VTEXPTS2LaTeX.exe, 29.03.2016 14:29
\publishedonline{4 April 2016}

\end{frontmatter}

\section{Introduction}\label{i}

Let $\{\xi_1,\xi_2,\ldots\}$ be a sequence of independent random
variables (r.v.s) with distribution functions (d.f.s) $\{F_{\xi
_1},F_{\xi_2},\ldots\}$, and let $\eta$ be a counting r.v., that~is,
an integer-valued, nonnegative, and nondegenerate at zero r.v. In
addition, suppose that the r.v. $\eta$ and r.v.s $\{\xi_1,\xi_2,\ldots
\}$ are independent. Let
$S_0=0$ and $S_n=\xi_1+\xi_2+\cdots+\xi_n$, $n\in\mathbb{N}$, be the
partial sums, and let\vadjust{\eject}
\begin{equation*}
S_\eta=\sum\limits_{k=1}^{\eta}\xi_k
\end{equation*}
be the random sum of $\{\xi_1,\xi_2,\ldots\}$.

We are interested in conditions under which the d.f.\ of $S_\eta$
\[
F_{S_\eta}(x)=\mathbb{P}(S_\eta\leqslant x)=\sum\limits_{n=0}^{\infty
}\mathbb{P}(\eta=n)\mathbb{P}(S_n\leqslant x)
\]
belongs to the class of $\mathcal{O}$-\textit{exponential} distributions.

{According to Albin and Sunden }\cite{al+sun} or Shimura and Watanabe
\cite{shim-wat},\textit{ a d.f. $F$ belongs to the class of $\mathcal
{O}$-\textit{exponential} distributions $\mathcal{OL}$ if
\begin{equation*}
0<\liminf\limits_{x\rightarrow\infty}\frac{\overline{F}(x+a)}{\overline
{F}(x)}\leqslant\limsup\limits_{x\rightarrow\infty}\frac{\overline
{F}(x+a)}{\overline{F}(x)}<\infty
\end{equation*}
for all $a\in\mathbb{R}$}, where $\overline{F}(x)=1-F(x)$, $x\in\mathbb
{R}$, is the tail of a d.f. $F$.

Note that if $F\in\mathcal{OL}$, then $\overline{F}(x)>0$ for all $x\in
\mathbb{R}$.

It is obvious that a d.f. $F$ belongs to the class $\mathcal{OL}$ if
and only if
\begin{equation}\label{ii}
\limsup\limits_{x\rightarrow\infty}\frac{\overline{F}(x-1)}{\overline
{F}(x)}<\infty
\end{equation}
or, equivalently, if and only if
\begin{equation*}
\sup\limits_{x\geqslant0}\frac{\overline{F}(x-1)}{\overline
{F}(x)}<\infty.
\end{equation*}

The last condition shows that class $\mathcal{OL}$ is quite wide. We
further describe some more popular subclasses of $\mathcal{OL}$ for
which we will present some results on the random convolution of
distributions from these subclasses.

\textit{A d.f. $F$ is said to
belong to the class $\mathcal{L}$ of long-tailed d.f.s if for every
fixed $a>0$, we have}
\[
\lim\limits_{x\rightarrow\infty}\frac{\overline{F}(x+a)}{\overline{F}(x)}=1.
\]

\textit{A d.f. $F$ is said to belong to the class $\mathcal{L}(\gamma)$
of exponential distributions with some $\gamma>0$ if
for any fixed $a>0$, we have}
\[
\lim\limits_{x\rightarrow\infty}\frac{\overline{F}(x+a)}{\overline
{F}(x)}={\rm e}^{-a\gamma}.
\]

\textit{A d.f. $F$
belongs to the class $\mathcal{D}$ \textup{(or }has a dominatingly
varying tail\textup{)} if for every fixed $a\in(0,1)$, we have}
\[
\limsup\limits_{x\rightarrow\infty}\frac{\overline{F}(xa)}{\overline
{F}(x)}<\infty.
\]

\textit{A d.f. $F$ supported on the interval $[0,\infty)$ belongs to
the class $\mathcal{S}$ \textup{(or }is subexponential\textup{)}
if}
\[
\lim\limits_{x\rightarrow\infty}\frac{\overline{F*F}(x)}{\overline{F}(x)}=2,
\]
\textit{where, as usual, $*$ denotes the convolution of d.f.s.}\vadjust{\eject}

\textit{A d.f. $F$ supported on the interval $[0,\infty)$ belongs to
the class $\mathcal{S}^*$
\textup{(} or is strongly subexponential\textup{)}
if }
\[
\mu:=\int\limits_{[0,\infty)}x{\rm d} F(x)<\infty\quad  {\rm and } \quad
\int\limits_0^x\overline{ F}(x-y)\overline{F}(y){\rm d} y \mathop{\sim
}\limits_{x\rightarrow\infty} 2\mu\overline{F}(x).
\]

If a d.f. $F$ is supported
on $\mathbb{R}$, then $F$ belongs to some of the classes $\mathcal{S}$
or $\mathcal{S}^*$
if $F^+(x) = F(x)\ind_{\{[0,\infty)\}}(x)$ belongs to the
corresponding class.

The presented definitions, together with Lemma 2 of Chistyakov \cite
{chist}, Lemma 9 of Denisov et al. \cite{den+}, Lemma 1.3.5(a) of
Embrechts et al. \cite{EKM},
and Lemma 1 of Kaas and Tang \cite{kaas+tang}, imply that
\[
\mathcal{S}^*\subset\mathcal{S}\subset\mathcal{L}\subset\mathcal{OL},
\ \mathcal{D}\subset\mathcal{OL},\quad \bigcup\limits_{\gamma>0}\mathcal
{L}(\gamma)\subset\mathcal{OL}.
\]

Now we present a few known results on when the d.f. $F_{S_\eta}$
belongs to some class. The first result about subexponential
distributions was proved
by Embrechts and Goldie (Theorem 4.2 in \cite{em+gol}) and Cline
(Theorem 2.13 in \cite{cline}).

\begin{thm}\label{t1}
\textit{Let $\{\xi_1,\xi_2,\ldots\}$ be independent copies of a
nonnegative r.v. $\xi$ with subexponential d.f. $F_\xi$. Let $\eta$ be
a counting r.v. independent of $\{\xi_1,\xi_2,\ldots\}$. If $\mathbb
{E}(1+\delta)^\eta<\infty$
for some $\delta>0$, then $F_{S_{\eta}}\in\mathcal{S}$.}
\end{thm}

In the case of strongly subexponential d.f.s, the following result,
which involves weaker restrictions
on the r.v. $\eta$, can be derived from Theorem 1 of Denisov et al.
\cite{den+2010} and Corollary 2.36 of Foss et al. \cite{foss}.

\begin{thm}\label{t2} \textit{Let $\{\xi_1,\xi_2,\ldots\}$ be
independent copies of a nonnegative r.v. $\xi$ with strongly
subexponential d.f. $F_\xi$ and finite mean $\mathbb{E}\xi$. Let $\eta$
be a counting r.v.\ independent of $\{\xi_1,\xi_2,\ldots\}$.
If $\mathbb{P}(\eta>x/c)\mathop{=}\limits_{x\rightarrow\infty}
o(\overline{F}_\xi(x))$
for some $ c > \mathbb{E}\xi$, then $F_{S_{\eta}}\in\mathcal{S}^*$.}
\end{thm}

Similar results for classes $\mathcal{D}$, $\mathcal{L}$, and $\mathcal
{OL}$ can be found in the papers of Leipus and \v{S}iaulys \cite
{ls-2012} and Danilenko and \v{S}iaulys \cite{ds-2015}. We further
present Theorem 6 from \cite{ls-2012}.

\begin{thm}\label{t3} \textit{Let $\{\xi_1,\xi_2,\ldots\}$ be
independent r.v.s with common d.f. $F_\xi\in\mathcal{L}$, and let $\eta$ be
a counting r.v. independent of $\{\xi_1,\xi_2,\ldots\}$  having
d.f.\ $F_\eta$.
If $\overline{F}_\eta(\delta
x)\mathop{=}\limits_{x\rightarrow\infty}\break
o(\sqrt{x}\ \overline{F}_\xi(x))$ for each $\delta\in(0,1)$, then
$F_{S_{\eta}}\in\mathcal{L}$.}
\end{thm}

In all presented results, r.v.s $\{\xi_1,\xi_2,\ldots\}$ are
identically distributed. In this work, we consider independent, but not
necessarily identically distributed, r.v.s. As was noted, we restrict
our consideration on the class $\mathcal{OL}$. In fact, in this paper,
we generalize the results of \cite{ds-2015}. If $\{\xi_1,\xi_2,\ldots\}
$ may be not identically distributed, then various collections of
conditions on r.v.s $\{\xi_1,\xi_2,\ldots\}$ and $\eta$ imply that
$F_{S_\eta}\in\mathcal{OL}$.
The rest of the paper is organized as follows. In Section~\ref{m}, we
formulate our main results. In Section~\ref{au}, we present all
auxiliary assertions, and the detailed proofs of the main results are
presented in Section~\ref{p}. Finally, a few examples of $\mathcal
{O}$-exponential random sums are described in Section~\ref{e}.

\section{Main results}\label{m}

In this section, we formulate our main results. The first result
describes the situation where the tails of d.f.s $F_{\xi_k}$ for large
indices $k$ are uniformly comparable with itself at the points $x$ and
$x-1$ for all $x\in[0,\infty)$.

\begin{thm}\label{tt1} \textit{Let $\{\xi_1,\xi_2, \ldots\}$ be
independent nonnegative random variables with d.f.s $\{F_{\xi_1},F_{\xi
_2},\ldots\}$, and let $\eta$ be a counting r.v. independent of $\{\xi
_1,\xi_2, \ldots\}$. Then $F_{S_\eta} \in\mathcal{OL}$ if the
following three conditions are satisfied.}

$\bullet$ \textit{For some $\kappa\in{\rm{supp}} (\eta)\setminus
\{0\}=\{n\in\mathbb{N}: \mathbb{P}(\eta=n)>0\}$, $F_{\xi_\kappa} \in
\mathcal{OL}$}.

$\bullet$ \textit{For each $k \in{\rm supp}(\eta)$, $k\leqslant
\kappa$, either $\lim\limits_{x\rightarrow\infty} \tfrac{\overline
{F}_{\xi_k}(x)}{\overline{F}_{\xi_\kappa}(x)}=0$ or $F_{\xi_k} \in
\mathcal{OL}$.}

$\bullet$ $\sup\limits_{x\geqslant0}\sup\limits_{k\geqslant
1}\tfrac{\overline{F}_{\xi_{\kappa+k}}(x-1)}{\overline{F}_{\xi_{\kappa
+k}}(x)}<\infty$.
\end{thm}

Since each d.f. from the class $\mathcal{OL}$ is comparable with
itself, the next assertion follows immediately from Theorem \ref{tt1}.

\begin{cor} \textit{Let $\{\xi_1,\xi_2, \ldots\}$ be independent
nonnegative random variables with common d.f. $F_\xi\in\mathcal{OL}$.
Then the d.f. of random sum $F_{S_\eta} $ is $\mathcal{O}$-exponential
for an arbitrary counting r.v. $\eta$.}
\end{cor}

Our second main assertion is dealt with counting r.v.s having finite support.

\begin{thm}\label{tt2} \textit{Let $\{\xi_1,\xi_2, \ldots, \xi_D \}$,
$D\in\mathbb{N}$, be independent nonnegative random variables with
d.f.s $\{F_{\xi_1},F_{\xi_2},\ldots F_{\xi_D}\}$, and let $\eta$ be a
counting r.v. independent of $\{\xi_1,\break\xi_2, \ldots, \xi_D \}$. Then
$F_{S_\eta} \in\mathcal{OL}$ under the following three conditions.}

$\bullet$ \textit{$\mathbb{P} (\eta\leqslant D)=1$.}

$\bullet$ \textit{For some $\kappa\in{\rm{supp}}(\eta)\setminus
\{0\}$, $F_{\xi_\kappa} \in\mathcal{OL}$.}

$\bullet$ \textit{For each $k \in\{1,2,\dots, D\}$, either
$\lim\limits_{x\rightarrow\infty} \frac{\overline{F}_{\xi
_k}(x)}{\overline{F}_{\xi_\kappa}(x)}=0$ or $F_{\xi_k} \in\mathcal
{OL}$. }
\end{thm}

Our last main assertion describes the case where the tails of d.f.s
$F_{\xi_k}$ are comparable at $x$ and $x-1$ asymptotically and
uniformly with respect to large indices $k$. In this case, conditions
are more restrictive for a counting r.v.

\begin{thm}\label{tt3}
\textit{Let $\{\xi_1,\xi_2, \ldots\}$ be independent nonnegative
random variables with d.f.s $\{F_{\xi_1},F_{\xi_2},\ldots\}$, and let
$\eta$ be a counting r.v. d.f. $F_\eta$ independent of $\{\xi_1,\xi_2,
\ldots\}$. Then $F_{S_\eta} \in\mathcal{OL}$ if the following five
conditions are satisfied.}

$\bullet$ {For some $\kappa\in{\rm{supp}} (\eta)\setminus\{0\}
$, $F_{\xi_\kappa} \in\mathcal{OL}$}.

$\bullet$ \textit{For each $k \in{\rm supp}(\eta)$, $k\leqslant
\kappa$, either $\lim\limits_{x\rightarrow\infty} \frac{\overline
{F}_{\xi_k}(x)}{\overline{F}_{\xi_\kappa}(x)}=0$ or $F_{\xi_k} \in
\mathcal{OL}$.}

$\bullet$ $\limsup\limits_{x\rightarrow\infty}\sup\limits
_{k\geqslant1}\frac{\overline{F}_{\xi_{\kappa+k}}(x-1)}{\overline
{F}_{\xi_{\kappa+k}}(x)}<\infty$.

$\bullet$ $\limsup\limits_{k\rightarrow\infty}\frac{1}{k}\sum_{l=1}^{k}\sup\limits_{x\geqslant0}
(\overline{F}_{\xi_{\kappa+l}}(x-1)-\overline{F}_{\xi_{\kappa
+l}}(x))<1$.

$\bullet$ For each $\delta\in(0,1)$, $\overline{F}_\eta(\delta
x)=O(\sqrt{x}\,\overline{F}_{\xi_\kappa}(x))$.
\end{thm}

\vspace*{-6pt}\section{Auxiliary lemmas}
\label{au}

In this section, we present all assertions that we use in the proofs of
our main results. We present some of auxiliary results with proofs.
The first assertion can be found in \cite{em+gol-1980} (see
Eq.~(2.12)).\vadjust{\eject}

\begingroup
\abovedisplayskip=5pt
\belowdisplayskip=5pt
\begin{lemma}\label{aul1} \textit{
Let $F$ and $G$ be two d.f.s satisfying $\overline{F}(x)>0$, $\overline
{G}(x)>0$, $x \in\mathbb{R}$. Then
\begin{eqnarray*}
\frac{\overline{F*G}(x-t)}{\overline{F*G}(x)} \leqslant\max\bigg\{
\sup\limits_{y \geqslant v} \frac{\overline{F}(y-t)}{\overline
{F}(y)},\sup\limits_{y \geqslant x-v+t} \frac{\overline
{G}(y-t)}{\overline{G}(y)}\bigg\}
\end{eqnarray*}
for all $x \in\mathbb{R}$, $v \in\mathbb{R}$, and $t>0$.}
\end{lemma}

The following assertion is the well-known Kolmogorov--Rogozin
inequality for concentration functions. Recall that the L\'{e}vy
concentration function or simply concentration function of a r.v.\ $X$
is the function
\[
Q_X(\lambda)=\sup\limits_{x\in\mathbb{R}}\mathbb{P}(x\leqslant
X\leqslant x+\lambda),\quad\lambda\in[0,\infty).
\]

The proof of the next lemma can be found in \cite{petrov} (Theorem 2.15).

\begin{lemma}\label{aul2} \textit{Let $ X_1,X_2,\ldots, X_n$ be
independent r.v.s, and let $Z_n=\sum_{k=1}^{n}X_k$. Then, for all $n\in
\mathbb{N}$,
\[
Q_{Z_n}(\lambda)\leqslant A\lambda\Bigg\{\sum\limits_{k=1}^{n}\lambda
_k^2\big(1-Q_{X_k}(\lambda_k)\big)\Bigg\}^{-1/2},
\]
where $A$ is an absolute constant, and $0<\lambda_k\leqslant\lambda$
for each $k\in\{1,2,\ldots,n\}$.}
\end{lemma}

The following assertion describes sufficient conditions under which the
d.f. of two independent r.v.s belongs to the class $\mathcal{OL}$.

\begin{lemma}\label{aul3} \textit{Let $X_1$ and $X_2$ be independent
r.v.s with d.f.s $F_{X_1}$ and $F_{X_2}$, respectively. Then the d.f.
$F_{X_1}*F_{X_2}$ of the sum $X_1+X_2$ is $\mathcal{O}$-exponential if
$F_{X_1}\in\mathcal{OL}$ and one of the following two conditions
holds:
\begin{eqnarray}
&\bullet&\lim\limits_{x\rightarrow\infty}\frac{\overline
{F}_{X_2}(x)}{\overline{F}_{X_1}(x)}=0\label{a1},\\
&\bullet& F_{X_2}\in
\mathcal{OL}.\nonumber
\end{eqnarray}
}
\end{lemma}
\endgroup
\begin{proof} We split the proof into three parts.

\begingroup
\abovedisplayskip=5pt
\belowdisplayskip=5pt
\textbf{I.} First, suppose that $\mathbb{P}(X_2\leqslant D)=1$
for some $D>0$. In this case, condition \eqref{a1} holds evidently.

For each real $x$, we have
\begin{align*}
\overline{F_{X_1}*F_{X_2}}(x)&=\mathbb{P}(X_1+X_2>x)= \int\limits_{(-\infty,D]} \overline{F}_{X_1}(x-y){\mbox{d}}F_{X_2}(y).
\end{align*}
Hence, for such $x$,
\begin{align*}
\frac{\overline{F_{X_1}*F_{X_2}}(x-1)}{\overline{F_{X_1}*F_{X_2}}(x)}&=\frac{\int_{(-\infty,D]} \overline{F}_{X_1}(x-1-y)\frac{\overline{F}_{X_1}(x-y)}{\overline{F}_{X_1}(x-y)}\mbox{d}F_{X_2}(y)}{\int_{(-\infty,D]} \overline{F}_{X_1}(x-y)\mbox{d}F_{X_2}(y)} \\
&\leqslant \frac{\int_{(-\infty,D]} \sup\limits_{y\leqslant D}\frac{\overline{F}_{X_1}(x-1-y)}{\overline{F}_{X_1}(x-y)}{\overline{F}_{X_1}(x-y)}\mbox{d}F_{X_2}(y)}{\int_{(-\infty,D]} \overline{F}_{X_1}(x-y)\mbox{d}F_{X_2}(y)}\\
& =  \sup_{z \geqslant x-D}\frac{\overline{F}_{X_1}(z-1)}{\overline{F}_{X_1}(z)}.
\end{align*}
\endgroup
This estimate implies that
\begin{align*}
\limsup_{x \rightarrow\infty}\frac{\overline{F_{X_1}*F_{X_2}}(x-1)}{\overline{F_{X_1}*F_{X_2}}(x)}&\leqslant\limsup_{x \rightarrow\infty} \sup_{z \geqslant x-D}\frac{\overline{F}_{X_1}(z-1)}{\overline{F}_{X_1}(z)} \\
&= \limsup_{y \rightarrow\infty} \frac{\overline{F}_{X_1}(y-1)}{\overline{F}_{X_1}(y)} \\
&<\infty
\end{align*}
because $F_{X_1}\in\mathcal{OL}$. So, $F_{X_1}*F_{X_2} \in\mathcal
{OL}$ as well.

\textbf{II.} Now let us consider the case where condition \eqref
{a1} holds but $\overline{F}_{X_2}(x)>0$ for all $x\in\mathbb{R}$.
For each real $x$, we have
\begin{align*}
\overline{F_{X_1}*F_{X_2}}(x)&=\int\limits_{-\infty}^{\infty} \overline
{F_{X_1}}(x-y)\mbox{d}F_{X_2}(y).
\end{align*}
Therefore,
\begin{align*}
 \overline{F_{X_1}*F_{X_2}}(x-1)&=\bigg(\, \int\limits_{(-\infty,\,
x-M]}+ \int\limits_{(x-M,\infty)} \bigg) \overline{F}_{X_1}(x-1-y)\mbox
{d}F_{X_2}(y)\\
&{\leqslant} \int\limits_{(-\infty,\,x-M]}\overline{F}_{X_1}(x-1-y)\frac
{\overline{F}_{X_1}(x-y)}{\overline{F}_{X_1}(x-y)}\mbox
{d}F_{X_2}(y)+\overline{F}_{X_2}(x-M) \\
&{\leqslant} \sup_{z \geqslant M}\frac{\overline
{F}_{X_1}(z-1)}{\overline{F}_{X_1}(z)}\!\int\limits_{(-\infty,\,
x-M]}\!\overline{F}_{X_1}(x-y)\mbox{d}F_{X_2}(y)+\overline{F}_{X_2}(x-M)
\end{align*}
for all $M,x$ such that $0<M<x-1$.
In addition, for such $M$ and $x$, we obtain
\begin{align*}
\overline{F_{X_1}*F_{X_2}}(x)&\geqslant \int\limits_{(-\infty,\,x-M]}\overline{F}_{X_1}(x-y)\mbox{d}F_{X_2}(y),\\
\overline{F_{X_1}*F_{X_2}}(x)&\geqslant\int\limits_{(M,\infty)}\overline{F}_{X_1}(x-y)\mbox{d}F_{X_2}(y) \\
&\geqslant\overline{F}_{X_1}(x-M)\overline{F}_{X_2}(M).
\end{align*}
The obtained estimates imply that
\begin{eqnarray*}
\frac{\overline{F_{X_1}*F_{X_2}}(x-1)}{\overline{F_{X_1}*F_{X_2}}(x)}
\leqslant\sup_{z \geqslant M}\frac{\overline{F}_{X_1}(z-1)}{\overline
{F}_{X_1}(z)}+\frac{\overline{F}_{X_2}(x-M)}{\overline
{F}_{X_1}(x-M)\overline{F}_{X_2}(M)}
\end{eqnarray*}
for all $x$ and $M$ such that $0<M<x-1$. Consequently,%\vadjust{\eject}
\begin{align*}
  \limsup_{x \rightarrow\infty} \frac{\overline
{F_{X_1}*F_{X_2}}(x-1)}{\overline{F_{X_1}*F_{X_2}}(x)}
& \leqslant
\sup_{z \geqslant M}\frac{\overline{F}_{X_1}(z-1)}{\overline
{F}_{X_1}(z)}+\frac{1}{\overline{F}_{X_2}(M)}\limsup_{x \rightarrow
\infty}\frac{\overline{F}_{X_2}(x-M)}{\overline{F}_{X_1}(x-M)}\\
&\quad = \sup_{z \geqslant M}\frac{\overline{F}_{X_1}(z-1)}{\overline{F}_{X_1}(z)}
\end{align*}
for all positive $M$.
Therefore,
\begin{equation*}
\limsup_{x \rightarrow\infty} \frac{\overline
{F_{X_1}*F_{X_2}}(x-1)}{\overline{F_{X_1}*F_{X_2}}(x)} \leqslant\limsup
_{M\rightarrow\infty}\frac{\overline{F}_{X_1}(M-1)}{\overline
{F}_{X_1}(M)}<\infty
\end{equation*}
because $F_{X_1}$ is $\mathcal{O}$-exponential. Consequently,
$F_{X_1}*F_{X_2} \in\mathcal{OL}$ by \eqref{ii}.

\textbf{III.} It remains to prove the assertion when both d.f.s
$F_{X_1}$ and $F_{X_2}$ are $\mathcal{O}$-exponen\-tial. By Lemma \ref
{aul1} we have
\begin{eqnarray*}
\frac{\overline{F_{X_1}*F_{X_2}}(x-1)}{\overline{F_{X_1}*F_{X_2}}(x)}
\leqslant\max\bigg\{\sup\limits_{z \geqslant M} \frac{\overline
{F}_{X_1}(z-1)}{\overline{F}_{X_1}(z)},\sup\limits_{z \geqslant x-M+1}
\frac{\overline{F}_{X_2}(z-1)}{\overline{F}_{X_2}(z)}\bigg\}
\end{eqnarray*}
for all $x$ and $M$ such that $0<M<x-1$. Therefore, for every positive $M$,
\begin{align*}
&\limsup_{x \rightarrow\infty} \frac{\overline
{F_{X_1}*F_{X_2}}(x-1)}{\overline{F_{X_1}*F_{X_2}}(x)}\\
&\quad\leqslant \max\bigg\{\sup\limits_{z \geqslant M}
\frac{\overline{F}_{X_1}(z-1)}{\overline{F}_{X_1}(z)},\limsup_{x
\rightarrow\infty}\sup\limits_{z \geqslant x-M+1} \frac{\overline
{F}_{X_2}(z-1)}{\overline{F}_{X_2}(z)}\bigg\} \\
& \quad= \max\bigg\{\sup\limits_{z \geqslant M} \frac
{\overline{F}_{X_1}(z-1)}{\overline{F}_{X_1}(z)},\limsup_{y \rightarrow
\infty} \frac{\overline{F}_{X_2}(y-1)}{\overline{F}_{X_2}(y)}\bigg\}.
\end{align*}
Letting $M$ tend to infinity, we get that
\begin{align*}
&\limsup_{x \rightarrow\infty} \frac{\overline
{F_{X_1}*F_{X_2}}(x-1)}{\overline{F_{X_1}*F_{X_2}}(x)}\\
&\quad\leqslant\max\bigg\{\limsup\limits_{M \rightarrow
\infty} \frac{\overline{F}_{X_1}(M-1)}{\overline{F}_{X_1}(M)},\limsup
_{y \rightarrow\infty} \frac{\overline{F}_{X_2}(y-1)}{\overline
{F}_{X_2}(y)}\bigg\}<\infty
\end{align*}
because $F_{X_1}$ and $F_{X_2}$ belong to class $\mathcal{OL}$.
Consequently, $F_{X_1}*F_{X_2} \in\mathcal{OL}$ due to requirement
\eqref{ii}. Lemma \ref{aul3} is proved.
\end{proof}

\begin{lemma}\label{aul4} \textit{Let $\{X_1, X_2,\ldots,X_n\}$ be
independent nonnegative r.v.s with d.f.s $\{F_{X_1},\break F_{X_2},\ldots,
F_{X_n}\}$. Let $F_{X_1} \in\mathcal{OL}$ and suppose that, for each
$k \in\{2,3,\dots,n\}$, either $\lim\limits_{x \rightarrow\infty
}\frac{\overline{F}_{X_k}(x)}{\overline{F}_{X_1}(x)}=0$ or
$F_{X_k} \in\mathcal{OL}$. Then the d.f. $F_{X_1}*F_{X_2}* \cdots*
F_{X_n}$ belongs to the class $ \mathcal{OL}$.}
\end{lemma}

\begin{proof}
We use induction on $n$.
If $n=2$, then the statement follows from Lemma \ref{aul3}. Suppose
that the statement holds if $n=m$, that is, $F_{X_1}*F_{X_2}*\cdots*
F_{X_m} \in\mathcal{OL}$, and we will show that the statement is
correct for $n=m+1$.

Conditions of the lemma imply that $F_{X_{m+1}} \in\mathcal{OL}$ or
\begin{align*}
&\lim_{x \rightarrow\infty} \frac{\overline{F}_{X_{m+1}}(x)}{\overline
{F_{X_1}*F_{X_2}*\cdots*F_{X_m}}(x)}
 = \lim_{x \rightarrow\infty} \frac{\overline
{F}_{X_{m+1}}(x)}{\mathbb{P}(X_1+\cdots+X_m>x)} \\
&\quad\leqslant \lim_{x \rightarrow\infty} \frac{\overline
{F}_{X_{m+1}}(x)}{\mathbb{P}(X_1>x)} =  \lim_{x \rightarrow\infty} \frac{\overline
{F}_{X_{m+1}}(x)}{\overline{F}_{X_1}(x)} =0.
\end{align*}
So, using Lemma \ref{aul3} again, we get
\begin{eqnarray*}
F_{X_1}*F_{X_2}*\cdots*F_{X_{m+1}}=(F_{X_1}*F_{X_2}*\cdots
*F_{X_m})*F_{X_{m+1}}\in\mathcal{OL}.
\end{eqnarray*}
We see that the statement of the lemma holds for $n=m+1$ and,
consequently, by induction, for all $n\in\mathbb{N}$. The lemma is proved.
\end{proof}

\section{Proofs of the main results}\label{p}
\begingroup\abovedisplayskip=12pt\belowdisplayskip=12pt
In this section, we present proofs of our main results.\medskip

\noindent\textbf{Proof of Theorem \ref{tt1}.}
Conditions of Theorem and Lemma \ref{aul4} imply that the d.f.\break
$F_{S_\kappa}(x)=\mathbb{P}(S_\kappa\leqslant x)$ belongs to the class
$ \mathcal{OL}$. So, we have
\begin{eqnarray}\label{f1}
\limsup\limits_{x\rightarrow\infty}\frac{\overline{F}_{S_{\kappa
}}(x-1)}{\overline{F}_{S_{\kappa}}(x)}<\infty
\end{eqnarray}
or, equivalently,\vspace{-6pt}
\begin{eqnarray}\label{f2}
\sup\limits_{x\geqslant0}\frac{\overline{F}_{S_{\kappa
}}(x-1)}{\overline{F}_{S_{\kappa}}(x)}\leqslant c_1
\end{eqnarray}
for some positive constant $c_1$.

We observe that, for all $x \geqslant0$,
\begin{eqnarray}\label{f3}
\frac{\mathbb{P}(S_\eta> x-1)}{\mathbb{P}(S_\eta> x)}=\mathcal
{J}_1(x)+\mathcal{J}_2(x),
\end{eqnarray}
where
\begin{eqnarray*}
\mathcal{J}_1(x)=\frac{\mathbb{P}(S_\eta> x-1,\eta\leqslant\kappa
)}{\mathbb{P}(S_\eta> x)}, \\
\mathcal{J}_2(x)=\frac{\mathbb{P}(S_\eta> x-1,\eta> \kappa)}{\mathbb
{P}(S_\eta> x)}.
\end{eqnarray*}
Since $\kappa\in{\rm{supp}}(\eta)$, we obtain
\begin{align*}
\mathcal{J}_1(x)&=\frac{\sum_{n=0}^{\kappa}\mathbb{P}(S_n >
x-1) \mathbb{P}(\eta= n)}{\sum_{n=0}^{\infty}\mathbb{P}(S_n >
x)\mathbb{P}(\eta= n)} \\
&\leqslant \frac{1}{\mathbb{P}(S_\kappa> x)\mathbb{P}(\eta= \kappa
)}\sum\limits_{n=0}^{\kappa}\mathbb{P}(S_\kappa> x-1) \mathbb{P}(\eta
= n) \\
&=\frac{\mathbb{P}(S_\kappa> x-1)}{\mathbb{P}(S_\kappa> x)} \frac
{\mathbb{P}(\eta\leqslant\kappa)}{ \mathbb{P}(\eta= \kappa)}.
\end{align*}\endgroup
Hence, it follows from \eqref{f1} that
\begin{eqnarray}\label{f4}
\limsup\limits_{x\rightarrow\infty}\mathcal{J}_1(x)<\infty.
\end{eqnarray}
By Lemma \ref{aul1} we have
\begin{eqnarray}\label{f5}
\frac{\mathbb{P}(S_{\kappa+1} > x-1)}{\mathbb{P}(S_{\kappa+1} >
x)}\leqslant\max\bigg\{\sup\limits_{z \geqslant M} \frac{\mathbb
{P}(S_{\kappa} > z-1)}{\mathbb{P}(S_{\kappa} > z)}, \sup\limits_{z
\geqslant x - M +1} \frac{\overline{F}_{\xi_{\kappa+1}}(z-1)}{\overline
{F}_{\xi_{\kappa+1}}(z)}\bigg\}
\end{eqnarray}
for all real $x$ and $M$.

The third condition of the theorem implies that
\begin{eqnarray}\label{f6}
\sup\limits_{x \geqslant0} \frac{\overline{F}_{\xi_{\kappa
+k}}(x-1)}{\overline{F}_{\xi_{\kappa+k}}(x)} \leqslant c_2
\end{eqnarray}
for all $k\in\mathbb{N}$ and some positive $c_2$.

If we choose $M={x}/{2}$ in estimate \eqref{f5}, then, using \eqref
{f2}, we get
\begin{eqnarray}\label{oopa}
\sup\limits_{x \geqslant0} \frac{\mathbb{P}(S_{\kappa+1} >
x-1)}{\mathbb{P}(S_{\kappa+1} > x)}\leqslant\max\left\{c_1,c_2\right
\}:=c_3.
\end{eqnarray}
Applying Lemma \ref{aul1} again, we obtain
\begin{eqnarray*}
\frac{\mathbb{P}(S_{\kappa+2} > x-1)}{\mathbb{P}(S_{\kappa+2} >
x)}\leqslant\max\bigg\{\sup\limits_{z \geqslant M} \frac{\mathbb
{P}(S_{\kappa+1} > z-1)}{\mathbb{P}(S_{\kappa+1} > z)}, \sup\limits
_{z \geqslant x - M +1} \frac{\overline{F}_{\xi_{\kappa
+2}}(z-1)}{\overline{F}_{\xi_{\kappa+2}}(z)}\bigg\}.
\end{eqnarray*}
By choosing $M={x}/{2}$ we get from inequalities \eqref{f6} and \eqref
{oopa} that
\begin{eqnarray*}
\sup\limits_{x \geqslant0} \frac{\mathbb{P}(S_{\kappa+2} >
x-1)}{\mathbb{P}(S_{\kappa+2} > x)}\leqslant c_3.
\end{eqnarray*}
Continuing the process, we find
\begin{eqnarray*}
\sup\limits_{x \geqslant0} \frac{\mathbb{P}(S_{\kappa+k} >
x-1)}{\mathbb{P}(S_{\kappa+k} > x)}\leqslant c_3
\end{eqnarray*}
for all $k\in\mathbb{N}$. Therefore,%\vadjust{\eject}
\begin{align}\label{f8}
\mathcal{J}_2(x)&=\frac{1}{\mathbb{P}(S_\eta> x)} \sum_{k=1}^{\infty}
\mathbb{P}(S_{\kappa+k} > x-1) \mathbb{P}(\eta= \kappa+k) \nonumber
\\
&\leqslant \frac{c_3}{\mathbb{P}(S_\eta> x)} \sum_{k=1}^{\infty}
\mathbb{P}(S_{\kappa+k} > x) \mathbb{P}(\eta= \kappa+k) \nonumber\\
&\leqslant \frac{c_3 \mathbb{P}(S_\eta> x)}{\mathbb{P}(S_\eta> x)}=c_3
\end{align}
for all $x \geqslant0$.

The obtained relations \eqref{f3}, \eqref{f4}, and \eqref{f8} imply that
\begin{eqnarray*}
\limsup\limits_{x\rightarrow\infty} \frac{\mathbb{P}(S_\eta>
x-1)}{\mathbb{P}(S_\eta> x)} < \infty.
\end{eqnarray*}
Therefore, the d.f. $F_{S_\eta}$ belongs to the class $ \mathcal{OL}$
due to requirement \eqref{ii}. Theorem \ref{tt1} is proved.
\hfill$\square$

\medskip
\noindent\textbf{Proof of Theorem \ref{tt2}.} The statement of the theorem can
be derived from Theorem \ref{tt1} or proved directly.
We present the direct proof of Theorem \ref{tt2}.

It is evident that $S_k=\xi_\kappa+ \sum_{n=1,\ n\neq\kappa
}^{k} \xi_n$ for each $k \geqslant\kappa$. Hence, by Lemma \ref{aul4},
$F_{S_k} \in\mathcal{OL}$ for all $ \kappa\leqslant k\leqslant D$.

If $x \geqslant1$, then we have
\begin{align}\label{pp}
\frac{\mathbb{P}(S_\eta> x-1)}{\mathbb{P}(S_\eta> x)}&=\frac{\sum
_{n=1 \atop n\in\rm{supp} (\eta)}^{D} \mathbb{P}(S_n>x-1)
\mathbb{P} (\eta=n)}{\sum_{n=1 \atop n\in\rm{supp} (\eta)
}^{D} \mathbb{P}(S_n>x) \mathbb{P} (\eta=n)}\nonumber\\
&\leqslant \frac{\mathbb{P}(S_\kappa>x-1) \mathbb{P} (\eta\leqslant
\kappa) + \sum_{n=\kappa+1 \atop n \in\rm{supp} (\eta)}^{D}
\mathbb{P}(S_n>x-1) \mathbb{P} (\eta=n)} {\mathbb{P} (S_\kappa
>x)\mathbb{P} (\eta=\kappa)+\sum_{n=\kappa+1 \atop n\in\rm
{supp} (\eta)}^{D} \mathbb{P}(S_n>x) \mathbb{P} (\eta=n)}\nonumber\\
&\leqslant \max\bigg\{\frac{\mathbb{P}(S_\kappa>x-1)
\mathbb{P} (\eta\leqslant\kappa)} {\mathbb{P} (S_\kappa>x)\mathbb
{P}(\eta=\kappa)},\max\limits_{\kappa+1 \leqslant n \leqslant D \atop
n \in\rm{supp} (\eta)} \frac{\mathbb{P}(S_n > x-1)}{\mathbb{P}(S_n >
x)}\bigg\},
\end{align}
where in the last step we use the inequality
\[
\frac{a_1+a_2+\cdots+a_n}{b_1+b_2+\cdots+b_n}\leqslant\max\left\{\frac
{a_1}{b_1},\frac{a_2}{b_2},\ldots,\frac{a_n}{b_n}\right\},
\]
provided that $n\geqslant1$ and $a_i,b_i>0$ for $i\in\{1,2,\ldots,n\}$.

Since $F_{S_n} \in\mathcal{OL}$ for all $n \geqslant\kappa$, we get
from \eqref{pp} that
\begin{eqnarray}\label{pp0}
\limsup\limits_{x \rightarrow\infty}\frac{\mathbb{P}(S_\eta>
x-1)}{\mathbb{P}(S_\eta> x)}<\infty,
\end{eqnarray}
and the statement of Theorem \ref{tt2} follows.
\hfill$\square$

\medskip
\noindent\textbf{Proof of Theorem \ref{tt3}.} As usual, it suffices to prove
relation \eqref{pp0}.
If $x\geqslant0$, then we have
\begin{align}
\mathbb{P}(S_\eta>x)&=\sum\limits_{n=1}^{\infty}\mathbf
{P}(S_n>x)\mathbb{P}(\eta=n)\nonumber\\
&\geqslant\mathbb{P}(S_\kappa>x)\mathbb{P}(\eta=\kappa)\nonumber\\
&\geqslant\overline{F}_{\xi_\kappa}(x)\mathbb{P}(\eta=\kappa\xch{)}{))}.\label{ppp2}
\end{align}
Similarly, for $K\geqslant2$ and $x\geqslant2K$,
\begin{align}\label{opa}
\mathbb{P}(S_\eta>x-1)&=\sum\limits_{n=1}^{\kappa}\mathbf
{P}(S_n>x-1)\mathbb{P}(\eta=n)\nonumber\\
&\quad +\sum\limits_{1\leqslant k\leqslant(x-1)/(K-1)}\mathbf{P}(S_{\kappa
+k}>x-1)\mathbb{P}(\eta=\kappa+k)\nonumber\\\
&\quad +\sum\limits_{ k>(x-1)/(K-1)}\mathbf{P}(x-1<S_{\kappa+k}\leqslant
x)\mathbb{P}(\eta=\kappa+k)\nonumber\\\
&\quad +\sum\limits_{ k>(x-1)/(K-1)}\mathbf{P}(S_{\kappa+k}> x)\mathbb
{P}(\eta=\kappa+k)\nonumber\\\
&:=\mathcal{K}_1(x)+\mathcal{K}_2(x)+\mathcal{K}_3(x)+\mathcal{K}_4(x).
\end{align}
The distribution function $F_{S_\kappa}$ belongs to the class $\mathcal
{OL}$ due to Lemma \ref{aul4}. So, by estimate \eqref{f4} we
have
\begin{equation}\label{px1}
\limsup\limits_{x \rightarrow\infty}\frac{\mathcal{K}_1(x)}{\mathbb
{P}(S_\eta>x)}=\limsup\limits_{x \rightarrow\infty}\mathcal
{J}_1(x)<\infty.
\end{equation}
Now we consider the sum $\mathcal{K}_2(x)$. Since $F_{S_\kappa}$ is
$\mathcal{O}$-exponential, we have
\begin{equation*}
\sup_{x\geqslant0} \frac{\mathbb{P}(S_\kappa> x-1)}{\mathbb
{P}(S_\kappa> x)} \leqslant c_4
\end{equation*}
with some positive constant $c_4$.
On the other hand, the third condition of Theorem \ref{tt3} implies that
\begin{equation*}
\sup_{x\geqslant c_5}\frac{\overline{F}_{\xi_{\kappa+
k}}(x-1)}{\overline{F}_{\xi_{\kappa+k}}(x)}\leqslant c_6
\end{equation*}
for some constants $c_5>2$, $c_6>0$ and all $k\in\mathbb{N}$.

By Lemma \ref{aul1} (with $v=c_5$) we have
\[
\frac{\mathbb{P}(S_{\kappa+1} > x-1)}{\mathbb{P}(S_{\kappa+1} > x)}
\leqslant\max\bigg\{\sup_{z \geqslant x-c_5+1} \frac{\mathbb
{P}(S_{\kappa} > z-1)}{\mathbb{P}(S_{\kappa} > z)}, \sup_{z \geqslant
c_5} \frac{\overline{F}_{\xi_{\kappa+1}}(z-1)}{\overline{F}_{\xi
_{\kappa+1}}(z)}\bigg\}.
\]
Consequently,
\begin{equation*}
\sup_{x \geqslant c_5} \frac{\mathbb{P}(S_{\kappa+1} > x-1)}{\mathbb
{P}(S_{\kappa+1} > x)}\leqslant\max\left\{c_4,c_6\right\}:=c_7.
\end{equation*}
Applying Lemma \ref{aul1} again for the sum $S_{\kappa+2}=S_{\kappa
+1}+\xi_{\kappa+2}$ (with $v=x/2+1/2$), we get%\vadjust{\eject}
\begin{equation*}
\frac{\mathbb{P}(S_{\kappa+2} > x-1)}{\mathbb{P}(S_{\kappa+2} > x)}
\leqslant\max\bigg\{\sup_{z \geqslant\frac{x}{2}+\frac{1}{2} }\frac
{\mathbb{P}(S_{\kappa+1} > z-1)}{\mathbb{P}(S_{\kappa+1} > z)}, \sup_{z
\geqslant\frac{x}{2}+\frac{1}{2}} \frac{\overline{F}_{\xi_{\kappa
+2}}(z-1)}{\overline{F}_{\xi_{\kappa+2}}(z)}\bigg\}.
\end{equation*}
If $x \geqslant2(c_5-1)+1$, then ${x}/{2}+{1}/{2} \geqslant c_5$.
Therefore, by the last inequality we obtain that
\begin{eqnarray*}
\sup_{x \geqslant2(c_5-1)+1}\frac{\mathbb{P}(S_{\kappa+2} >
x-1)}{\mathbb{P}(S_{\kappa+2} > x)} \leqslant c_7.
\end{eqnarray*}
Applying Lemma \ref{aul1} once again (with $v=x/3+2/3$), we get
\begin{equation*}
\frac{\mathbb{P}(S_{\kappa+3} > x-1)}{\mathbb{P}(S_{\kappa+3} > x)}
\leqslant\max\bigg\{\sup_{z \geqslant\frac{2x}{3}+\frac{1}{3} }\frac
{\mathbb{P}(S_{\kappa+2} > z-1)}{\mathbb{P}(S_{\kappa+2} > z)}, \sup_{z
\geqslant\frac{x}{3}+\frac{2}{3}} \frac{\overline{F}_{\xi_{\kappa
+3}}(z-1)}{\overline{F}_{\xi_{\kappa+3}}(z)}\bigg\}.
\end{equation*}\eject
\noindent If $x \geqslant3(c_5-1)+1$, then
$
{2x}/{3}+{1}/{3} \geqslant 2(c_5-1)+1 $ and
${x}/{3}+{2}/{3} \geqslant c_5$.
So, the last estimate implies
\begin{eqnarray*}
\sup_{x \geqslant3(c_5-1)+1}\frac{\mathbb{P}(S_{\kappa+3} >
x-1)}{\mathbb{P}(S_{\kappa+3} > x)} \leqslant c_7.
\end{eqnarray*}
Continuing the process, we can get that
\begin{eqnarray}\label{opa1}
\sup_{x \geqslant k(c_5-1)+1}\frac{\mathbb{P}(S_{\kappa+k} >
x-1)}{\mathbb{P}(S_{\kappa+k} > x)} \leqslant c_7
\end{eqnarray}
for all $k\in\mathbb{N}$.

We can suppose that $K=c_5$ in representation \eqref{opa}. In such a
case, it follows from inequality \eqref{opa1} that
\begin{align}\label{pp2}
\limsup\limits_{x \rightarrow\infty}\frac{\mathcal{K}_2(x)}{\mathbb
{P}(S_\eta>x)}&\leqslant\limsup\limits_{x \rightarrow\infty} \frac
{c_7}{\mathbb{P}(S_\eta> x)}\sum\limits_{1 \leqslant k \leqslant\frac
{x-1}{c_5-1}}\mathbb{P}(S_{\kappa+k} > x)\mathbb{P}(\eta=\kappa
+k)\nonumber\\
&\leqslant c_7.
\end{align}
Since, obviously,
\begin{equation}\label{pp3}
\limsup\limits_{x \rightarrow\infty}\frac{\mathcal{K}_4(x)}{\mathbb
{P}(S_\eta>x)}\leqslant1,
\end{equation}
it remains to estimate sum $\mathcal{K}_3(x)$.
Using Lemma \ref{aul2}, we obtain
\begin{equation*}
\mathcal{K}_3(x)\leqslant A\sum\limits_{k>\frac{x-1}{c_5-1}}\mathbb
{P}(\eta=\kappa+k)\Bigg(\sum\limits_{l=1}^{k}\Big(1-\sup\limits_{x\in
\mathbb{R}}\mathbb{P}(x-1\leqslant\xi_{\kappa+l}\leqslant x)\Big)\Bigg)^{-1/2}
\end{equation*}
with some absolute positive constant $A$. By the fourth condition of
the theorem,
\begin{equation*}
\frac{1}{k}\sum\limits_{l=1}^{k}\,\sup\limits_{x\in\mathbb{R}}
\big(\overline{F}_{\xi_{\kappa+l}}(x-1)-\overline{F}_{\xi_{\kappa
+l}}(x)\big)\leqslant1-\Delta
\end{equation*}
for some $0<\Delta<1$ and all sufficiently large $k$. So, for such $k$,%\vadjust{\eject}
\begin{equation*}
\sum\limits_{l=1}^{k}\Big(1-\sup\limits_{x\in\mathbb{R}}\mathbb
{P}(x-1\leqslant\xi_{\kappa+l}\leqslant x)\Big)\geqslant k\Delta.
\end{equation*}
From the last estimate it follows that
\begin{align*}
\mathcal{K}_3(x)&\leqslant\frac{A}{\sqrt{\Delta}}\sum\limits_{k>\frac
{x-1}{c_5-1}}\frac{1}{\sqrt{k}}\mathbb{P}(\eta=\kappa+k)\\
&\leqslant \frac{A}{\sqrt{\Delta}}\sqrt{\frac{c_5-1}{x-1}}\mathbb
{P}\bigg(\eta>\kappa+\frac{x-1}{c_5-1}\bigg)
\end{align*}
for sufficiently large $x$. Therefore,
\begin{align}\label{ppp}
&\limsup\limits_{x \rightarrow\infty}\frac{\mathcal{K}_3(x)}{\mathbb
{P}(S_\eta>x)}\nonumber\\
&\quad \leqslant\frac{A}{\sqrt{\Delta}}\frac{\sqrt
{c_5-1}}{\mathbb{P}(\eta=\kappa)}
\limsup\limits_{x \rightarrow\infty}\frac{\overline{F}_\eta(\frac
{x-1}{c_5-1})}{\sqrt{x-1}\ \overline{F}_{\xi_\kappa}(x-1)}
\limsup\limits_{x \rightarrow\infty}\frac{\overline{F}_{\xi_\kappa
}(x-1)}{\overline{F}_{\xi_\kappa}(x)}\nonumber\\
&\quad <\infty
\end{align}
by estimate \eqref{ppp2} and the last condition of the theorem.
Representation \eqref{opa} and estimates \eqref{px1}, \eqref{pp2},
\eqref{pp3}, and \eqref{ppp} imply the desired inequality \eqref{pp0}.
Theorem \ref{tt3} is proved.
\hfill$\square$

\section{Examples of $\mathcal{O}$-exponential random sums}\label{e}

In this section, we present three examples of random sums $ S_{\eta}$
for which the d.f.s $F_{S_\eta}$ are $\mathcal{O}$-exponential.

\begin{ex}\label{ex1} Let $\{\xi_1,\xi_2,\ldots\}$ be
independent r.v.s. We suppose that the r.v. $\xi_k$ for $k\in\{
1,2,\ldots,D\}$ is distributed according to the Pareto law with
parameters~$k$ and $\alpha$, that is,
\[
\overline{F}_{\xi_{k}}(x)=\left(\frac{k}{k+x}\right)^\alpha,\quad x\geqslant0,
\]
where $k\in\{1,2,\ldots,D\}$, $D\geqslant1$, and $\alpha>0$.
In addition, we suppose that the r.v. $\xi_{D+k}$ for each $k\in\mathbb
{N}$ is distributed according to the exponential law with parameter
$\lambda/k$, that is,
\[
\overline{F}_{\xi_{D+k}}(x)={\rm e}^{-\lambda x/k},\quad x\geqslant0.
\]
\end{ex}

It follows from Theorem \ref{tt1} that the d.f. of the random sum
$S_\eta$ is $\mathcal{O}$-exponential for each counting r.v. $\eta$
independent of $\{\xi_1,\xi_2,\ldots\}$ under the condition $\mathbb
{P}(\eta=\kappa)>0$ for some $\kappa\in\{1,2,\ldots, D\}$ because:
\begin{eqnarray*}
&\bullet& F_{\xi_k}\in\mathcal{L}\subset\mathcal{OL}\quad  \mbox{for each}\
k\leqslant\kappa,\\
&\bullet& \sup\limits_{x\geqslant0}\sup\limits_{k\geqslant1}\frac
{\overline{F}_{\xi_{\kappa+k}}(x-1)}{\overline{F}_{\xi_{\kappa+k}}(x)}\\
&&\hspace{2mm}=\max\bigg\{ \sup\limits_{0\leqslant x\leqslant1}\sup
\limits_{k\geqslant1}\frac{1}{\overline{F}_{\xi_{\kappa+k}}(x)},\
\sup\limits_{x> 1}\sup\limits_{k\geqslant1}\frac{\overline{F}_{\xi
_{\kappa+k}}(x-1)}{\overline{F}_{\xi_{\kappa+k}}(x)}\bigg\}\\
&&\hspace{2mm}=\max\bigg\{ \sup\limits_{0\leqslant x\leqslant1}\max
\bigg\{\max\limits_{1\leqslant k\leqslant D-\kappa}\bigg(\frac{\kappa
+k+x}{\kappa+k}\bigg)^\alpha,\
\sup\limits_{k\geqslant1}\,{\rm e}^{\lambda x/k}\bigg\},\\
&&\quad \  \sup\limits_{ x> 1}\max\bigg\{\max\limits
_{1\leqslant k\leqslant D-\kappa}\bigg(\frac{\kappa+k+x}{\kappa
+k+x-1}\bigg)^\alpha,\
\sup\limits_{k\geqslant1}\,{\rm e}^{\lambda/k}\bigg\}\bigg\}\\
&&\hspace{2mm}\leqslant\max\big\{2^\alpha,\, {\rm e}^{\lambda}\big\}.
\end{eqnarray*}\eject

\begin{ex}\label{ex2} Let a r.v.\  $\eta$ be uniformly
distributed on $\{1,2,\ldots, D\}$, that is,
\[
\mathbb{P}(\eta=k)=\frac{1}{D},\quad  k\in\{1,2,\ldots,D\},
\]
for some $D\geqslant2$. Let $\{\xi_1,\xi_2,\ldots, \xi_D\}$ be
independent r.v.s, where $\xi_1$ is exponentially distributed, and $\xi
_2,\ldots,\xi_D$ are uniformly distributed.
\end{ex}

If the r.v. $\eta$ is independent of the r.v.s $\{\xi_1,\xi_2,\ldots,
\xi_D\}$, then Theorem \ref{tt2} implies that the d.f. of the random
sum $S_\eta$ is $\mathcal{O}$-exponential.

\begin{ex}\label{ex3} Let $\{\xi_1,\xi_2,\ldots\}$ be
independent r.v.s, where $\{\xi_1,\xi_2,\ldots,\xi_{\kappa-1}\}$ are
finitely supported, $\kappa\geqslant2$, and $\xi_\kappa$ is
distributed according to the Weibull law, that is,
\[
\overline{F}_{\xi_\kappa}(x)=\mbox{e}^{-\sqrt{x}},\quad  x\geqslant0.
\]
In addition, we suppose that the r.v. $\xi_{\kappa+k}$ for each
$k=m^2$, $m\geqslant2$, has the d.f. with tail
\[
\overline{F}_{\xi_{\kappa+k}}(x)=
\begin{cases}
1 & {\rm if}\ x<0,\\
\frac{1}{k}& {\rm if}\ 0\leqslant x<k,\\
\frac{1}{k}{\rm e}^{-(x-k)}& {\rm if}\ x\geqslant k,
\end{cases}
\]
whereas for each remaining index $k\notin\{m^2,\mbox{$m\in\mathbb{N}$}\setminus\{
1\}\}$, the r.v. $\xi_{\kappa+k}$ has the exponential distribution,
that is,
\[
\overline{F}_{\xi_{\kappa+k}}(x)={\rm e}^{-x},\quad  x\geqslant0.
\]
\end{ex}

If the counting r.v. $\eta$ is independent of $\{\xi_1,\xi_2,\ldots\}$
and is distributed according to the Poisson law with parameter $\lambda
$, then it follows from Theorem \ref{tt3} that the random sum $S_\eta$
is $\mathcal{O}$-exponentially distributed because:
\begin{eqnarray*}
&\bullet& F_{\xi_\kappa}\in\mathcal{L}\subset\mathcal{OL};\\
&\bullet& \lim\limits_{x\rightarrow\infty}\frac{\overline{F}_{\xi
_k}(x)}{\overline{F}_{\xi_\kappa}(x)}=0\quad  \mbox{if}\ k=1,2,\ldots, \kappa
-1;\\
&\bullet& \sup\limits_{x\geqslant1}\sup\limits_{k\geqslant1}\frac
{\overline{F}_{\xi_{\kappa+k}}(x-1)}{\overline{F}_{\xi_{\kappa+k}}(x)}\\
&&\quad= \sup\limits_{x\geqslant1}\max\bigg\{\sup_{k\geqslant
1,\,k=m^2,\,m\geqslant2}\frac{\overline{F}_{\xi_{\kappa
+k}}(x-1)}{\overline{F}_{\xi_{\kappa+k}}(x)},\,\sup_{k\geqslant1,\,
k\neq m^2}\frac{\overline{F}_{\xi_{\kappa+k}}(x-1)}{\overline{F}_{\xi
_{\kappa+k}}(x)}\bigg\}\\
&&\quad=\sup\limits_{x\geqslant1}\max\Big\{\sup_{k\geqslant
1,\,k=m^2,\,m\geqslant2}\big\{\ind_{[1,k)}(x)+\mbox{e}^{x-k}\ind
_{[k,k+1)}(x)+\mbox{e}\ind_{[k+1,\infty)}(x)\big\},\\ &&
\hspace{19mm}
\sup_{k\geqslant1,\,k\neq m^2}\mbox{e}\Big\}=\mbox{e};\\
&\bullet& \limsup\limits_{k\rightarrow\infty}\frac{1}{k}\sum\limits
_{l=1}^{k}\sup\limits_{x\geqslant0}\big(\overline{F}_{\xi_{\kappa
+l}}(x-1)-\overline{F}_{\xi_{\kappa+l}}(x)\big)\\
&&\quad = \limsup\limits_{k\rightarrow\infty}\frac{1}{k}\Bigg(\sum
\limits_{l=1,\,l=m^2}^{k}\bigg(1-\frac{1}{l}\bigg)+\bigg(1-\frac
{1}{\mbox{e}}\bigg)\sum\limits_{l=1,\,l\neq m^2}^{k}1\Bigg)\\
&&\quad  \leqslant\bigg(1-\frac{1}{\mbox{e}}\bigg);\\
&\bullet& \overline{F}_\eta(x)<\bigg(\frac{{\rm e} \lambda}{x}\bigg
)^x,\quad x>\lambda.
\end{eqnarray*}

Here the last estimate is the well-known Chernof bound for the Poisson
law (see, e.g., p.~97 in \cite{mm}).

As we can see, the r.v.s $\{\xi_1,\xi_2,\ldots\}$ from the last example
satisfy the conditions of Theorem \ref{tt3}, whereas the third
condition of Theorem \ref{tt1} does not hold because, in this case,
\begin{eqnarray*}
\sup\limits_{x\geqslant0}\sup\limits_{k\geqslant1}\frac{\overline
{F}_{\xi_{\kappa+k}}(x-1)}{\overline{F}_{\xi_{\kappa+k}}(x)}
\geqslant \sup\limits_{0 \leqslant x < 1}\sup\limits_{k\geqslant
1}\frac{\overline{F}_{\xi_{\kappa+k}}(x-1)}{\overline{F}_{\xi_{\kappa+k}}(x)}
\geqslant \sup\limits_{0 \leqslant x < 1}\  \sup\limits_{ k=m^2, m
\geqslant2} k = \infty.
\end{eqnarray*}

\section*{Acknowledgments} We would like to thank the anonymous
referees for the detailed and helpful comments on the first and second
versions of the manuscript.

% structpyb loaded by giedrius.virsilas, 2016-03-29 14:49:46

%
\end{document}